\documentclass[12pt,reqno]{amsart}
\usepackage{mathptmx}      %
\usepackage{xcolor}
\usepackage{amssymb}
\usepackage{amscd}
\usepackage{amsmath}
\usepackage{amsthm}
\usepackage{amstext}
\usepackage{amsrefs}
\usepackage{hyperref}
\usepackage[a4paper,tmargin=3cm,bmargin=3cm,lmargin=3cm,rmargin=3cm]{geometry}
\usepackage{amsrefs}
\usepackage{enumitem}
\usepackage{scrtime}
\newcommand{\supp}{{\rm supp \;}}

\newcommand{\PP}{\,\Pi\hspace{-11pt}\Pi\,}
\newcommand{\R}{\mathbb{R}}
\newcommand{\N}{\mathbb{N}}

\newcommand{\dd}{\mathrm{d}}
\newcommand{\ddd}{\,\text{\rm{\mbox{\dj}}}}
\newcommand{\bmo}{\mathrm{bmo}}
\newcommand{\BMO}{\mathrm{BMO}}

\newcommand{\norm}[1]{\left\Vert#1\right\Vert}
\newcommand{\brkt}[1]{\left(#1\right)}
\newcommand{\abs}[1]{\left|#1\right|}
\newcommand{\set}[1]{\left\{#1\right\}}
\newcommand{\esc}[1]{\langle{#1}\rangle}

\newcommand{\phase}{\varphi}
\renewcommand{\d}{\partial}
    \newtheorem{thm}{Theorem}[section]
    \newtheorem{cor}[thm]{Corollary}
    \newtheorem{lem}[thm]{Lemma}
    \newtheorem{prop}[thm]{Proposition}
    \newtheorem{defn}[thm]{Definition}
    \newtheorem{rem}[thm]{Remark}

    \numberwithin{equation}{section}

\begin{document}

\title[Bilinear paraproducts]{Some endpoint estimates for bilinear paraproducts and applications}
\keywords{}
\subjclass[2000]{}
\author[S.~Rodr\'iguez-L\'opez]{Salvador Rodr\'iguez-L\'opez}
\address{Department of Mathematics, Uppsala University, Uppsala, SE 75106,  Sweden}
\email{salvador@math.uu.se}
\urladdr{http://www.math.uu.se/~salvador}
\author[W.~Staubach]{Wolfgang Staubach}
\email{wulf@math.uu.se}
\urladdr{http://www.math.uu.se/~wulf}

\thanks{The first author has been partially supported by the Grant MTM2010-14946.}

\subjclass[2010]{Primary 35S05, 35S30, 35S50, 42B35; Secondary 42B15, 42B20, 42B25}

\keywords{Bilinear paraproducts, BMO-type spaces, Bilinear Fourier integral operators, Bilinear Multipliers, Kato-Ponce estimates}

\maketitle

\begin{abstract}
In this paper we establish the boundedness of bilinear paraproducts on local BMO spaces. As applications, we also investigate the boundedness of bilinear Fourier integral operators and bilinear Coifman-Meyer multipliers on these spaces and also obtain a certain end-point result concerning Kato-Ponce type estimates.
\end{abstract}

\section{Introduction}
This paper is mainly concerned with end-point estimates for bilinear paraproducts of the form
\begin{equation}\label{intro:const coeff para}
 \Pi(f,g)(x):= \int_0^\infty Q_t f(x)\, P_tg(x) m(t)\frac{\dd t}{t}
\end{equation}  
or
\begin{equation}\label{intro:var coeff para}
\PP(f,g)(x):= \int_0^1 Q_t f(x)\, P_tg(x) m(t,x)\frac{\dd t}{t},
\end{equation}
where $P_t$ and $Q_t$ are standard frequency localisation operators. More specifically, we are interested in studying the behaviour of these paraproducts when the functions $f$ and $g$ belong to various local or global BMO classes. In this connection the basic results are due to L. Grafakos and R. Torres \cite{GT} which encompass the main end-point estimates regarding the boundedness of multilinear paraproducts. In particular Grafakos and Torres show the $L^\infty\times L^\infty \to \BMO$ boundedness of bilinear Calder\'on-Zygmund operators, but the case where the functions $f$ and $g$ belong to BMO-type spaces is not covered by \cite{GT}, or any other investigations that we are aware of.
As a matter of fact, this paper stems from our investigation of the problem of boundedness of bilinear Fourier integral operators \cite{RRS}, where it was shown that there exist amplitudes $\sigma(x,\xi,\eta)\in S^{0}_{1,0}(1,2)$ and non-degenerate phase functions $\phase_1$ and $\phase_2$ (see Definitions \ref{defn of hormander symbols} and \ref{defn:phase}) for which the associated one dimensional bilinear Fourier integral operator given by
\begin{equation*}
T^{\phase_1,\phase_2}_{\sigma}(f,g)(x) = \iint \sigma(x,\xi,\eta)\widehat{f}(\xi)\widehat{g}(\eta)e^{i\phase_1(x,\xi)+i \phase_2(x,\eta)} \ddd\xi \ddd\eta,
\end{equation*}
 fails to be bounded from $L^\infty\times L^\infty\to \BMO$. This rather surprising fact prompted us to search for alternative spaces for which a modification of the aforementioned negative result is valid. However, in doing so, we soon entered a rather unexplored territory which included at one end, the study of certain endpoint estimates for bilinear paraproducts about which little was known, and at the other end, the study of exotic function spaces which didn't exist in the literature.

Therefore, we had to deal with various issues which were not a-priori related to the study of multilinear operators. However as a bi-product, our investigation yields a new characterisation for the local BMO space (Theorem \ref{cor:bmo}). The definition of the new function spaces (see e.g. Definitions \ref {defn: BMO_w} and \ref{defn: X_w}) enable us to prove the boundedness of bilinear paraproducts of the form \eqref{intro:const coeff para} and \eqref{intro:var coeff para} which have been established in Theorem \ref{thm:main} and Theorem \ref{xdependent version} respectively. Since we needed the boundedness of linear Fourier integral operators on local BMO and we were not able to locate such a result anywhere in the literature, a proof was provided in Theorem \ref{eq:thm_bmo}. This theorem is also used in proving one of our main results concerning the boundedness of one dimensional bilinear Fourier integral operators (Theorem \ref{thm:1d_FIO}).\\

The paper also deals with two other issues, the first is the bilinear Coifman-Meyer multipliers, and the second is the problem of end-point Kato-Ponce estimates. From the point of view of this paper, these two problems are intimately connected. The problem of finding the least number of derivatives required for the validity of boundedness of the bilinear multipliers on one hand, and that of the Kato-Ponce estimate on the other, has been intensively investigated by many authors. Here we only mention those that have been of particular importance and interest to us, which are the papers by N. Tomita \cite{Tomita}, A. Miyachi and N. Tomita \cite{MT} and those of L. Grafakos, D. Maldonado and V. Naibo \cite{GMN}, and L. Grafakos and S. Oh \cite{GOh}. We also got the opportunity of applying our results to the $\bmo$ related end-point estimates for bilinear Coifman-Meyer multipliers (Theorem \ref{thm:H-M}) and Kato-Ponce estimates (Theorem \ref{thm:kato_ponce}).\\

The structure of the paper is as follows. Section \ref{prem} introduces some notations and recalls the definition of some standard function spaces. In Section \ref{functionspaces} we introduce our new function spaces, which are used in the formulation of our main results. Section \ref{local estimates} contains estimates for localization operators, where to our knowledge, the estimate for the operator $P_t$ in Proposition \ref{eq:pt_qt} is new. In Section \ref{paraproducts} we state and prove our main result concerning bilinear paraproducts and also obtain a variable coefficient version thereof in Subsection \ref{xdependent}. In Section \ref{FIOs} the boundedness of one dimensional bilinear Fourier integral operators on BMO-type spaces is proven. Section \ref{kato-ponce} is devoted to the study of bilinear Coifman-Meyer multipliers and our version of the end-point Kato-Ponce estimate.

\section{Preliminaries}\label{prem}

In what follows, we use the notation
\[
	a(tD)f(x):=\int_{\R^n} a(t\xi) \widehat{f}(\xi)e^{ix\cdot \xi} \, \ddd\xi,
\]
for $t>0$ and $a$ in a suitable symbol class, where $ \ddd\xi$ denotes the Lebesgue measure in $\mathbb{R}^n$ normalised by $(2\pi)^{-n}$ . Whenever $t=1$ we shall simply write $a(D)$.

Here and in the rest of the paper, the notation $A\lesssim B$ means that there exist a constant $C$ such that $A\leq C B .$ The notation $A\thickapprox B$ stands for $A\lesssim B$ and $B\lesssim A$.

Given a bump function $\widehat{\Theta},$ such that $\widehat{\Theta}=1$ in a neighbourhood of the origin (where $\widehat{\;}$ denotes the Fourier transform),  the Hardy space $H^1$ is the class of tempered distributions $f$ such that
\begin{equation}\label{eq:H1}
	\norm{f}_{H^1}\thickapprox \int \sup_{t>0}\sup_{\abs{x-y}<t} \abs{\widehat{\Theta}(tD) f(y)}\dd x,
\end{equation}
see e.g. \cites{FS}.
The local Hardy space $h^1$ can be defined in a similar way as $H^1$ (taking the supremum on $0<t<1$ instead), but we will use the following characterisation of $h^1$;\\
A function $f\in h^1(\R^n)$ if and only if $f\in L^1(\R^n)$ and $r_j(f)\in L^1(\R^n)$ for any $j=1,\ldots, n.$ Here $r_j$ denotes the local Riesz transform, which is defined by
\begin{equation}\label{eq:local_Riesz}
	r_j(f)(x)=-i\int \frac{\xi_j}{\abs{\xi}}(1-\widehat{\Theta}(\xi)) \widehat{f}(\xi) e^{i x\cdot\xi}\ddd \xi.
\end{equation}
Moreover
\begin{equation}\label{eq:norm_h1}
	\norm{f}_{h^1(\R^n)}\thickapprox \norm{f}_{L^1}+\sum_{j=1}^n \norm{r_j( f)}_{L^1}.
\end{equation}

We refer the reader to the work of D.~Goldberg \cite{Gol} and the paper of C.~Fefferman and E.M.~Stein \cite{FS} for further properties of $h^1$ and $H^1$ respectively.

The dual of $H^1$ is the space of functions in $\BMO$ (see e.g \cite{FS} or \cite{G}). The dual of $h^1$ is the space $\bmo$ which is the class of locally integrable functions for which
\begin{equation}\label{eq:bmo}
	\begin{split}
	\norm{f}_{\bmo} &:=\norm{f}_{\BMO}+\norm{\widehat{\Theta}(D) f}_{L^\infty}\thickapprox\norm{(1-\widehat{\Theta}(D)) f}_{\BMO}+\norm{\widehat{\Theta}(D) f}_{L^\infty}<+\infty
	\end{split}
\end{equation}
see e.g.~\cites{Gol}. It is worth mentioning that the definitions above do not depend on the choice of the function $\Theta$ (i.e. different choices yield equivalent norms).
\section{Admissible weights and related function spaces}\label{functionspaces}
In our investigation, we will introduce some classes of function spaces depending on certain weights which are required to satisfy certain properties. To this end we define the following class of weights.
\begin{defn}
A positive weight function $w$ defined on $(0,\infty)$ is {\em admissible}, if it has all of the following properties:
\begin{enumerate}[itemindent=0pt,label=\rm\bfseries (A\arabic*), fullwidth]
\item\label{A1} For every $s>0$, $0< \mathfrak{i}(s):=\inf_{t>0} \frac{w(st)}{w(t)}\leq \sup_{t>0} \frac{w(st)}{w(t)}=:\mathfrak{s}(s)<\infty.$
\item\label{A2} For every $t>0$, $w(t)\geq 1$
\item\label{A3} For some $N>0$, $\sup_{t>0} w(t) \brkt{1+ \frac{1}{t}}^{-N}<\infty$.
\item\label{A4} For any closed interval $I\subset (0,\infty)$, $0<\inf_{s\in I} \mathfrak{i}(s)\leq \sup_{s\in I}\mathfrak{s}(s)<\infty.$

\end{enumerate}
\end{defn}

\subsection*{Examples of admissible weights}

\begin{enumerate}[fullwidth]
	\item It is trivial to see that $w(t)=1$ is admissible.
	\item The function $w(t)=1+\log_+ 1/t$ is admissible. Property \ref{A2} is clear from the definition. Moreover, since for any $\epsilon>0$, $\sup _{0<t<1} t^\epsilon w(t)$ is finite and $w$ is bounded for $t>1$, it follows that it satisfies \ref{A3}. Furthermore, it is easy to see that, for any $s\geq 1$ and $t>0$ one has  $(1+\log s)^{-1}\leq \frac{w(st)}{w(t)}\leq 1$.
	Then, if $s\leq 1$, since this particular weight is decreasing, $w(st)\geq w(t)=w((st)/s)\geq (1-\log s)^{-1} w(st)$,
	which yields $\mathfrak{i}(s)\geq  \chi_{(0,1)}(s)+(1+\log s)^{-1}\chi_{[1,\infty)}(s)$ and $\mathfrak{s}(s)\leq 1-\log s \chi_{(0,1)}(s)$. Hence \ref{A1} and \ref{A4} are also satisfied.
\item Given any $\alpha>0$, $w_\alpha(t)= \brkt{1+\log_+ 1/t}^\alpha$ is also admissible. 
\end{enumerate}

\begin{cor}\label{cor:equivalence} For any admissible weight $w$, one has that if $t\approx s$ then
$w(t)\approx w(s)$.
\end{cor}
\begin{proof}   The result  follows by combining properties \ref{A1} and \ref{A4}.
 \end{proof}

Let $\psi$ be a function such that $\widehat{\psi}\in \mathcal{C}^\infty_c(\R^n)$ such that $0\not\in \supp \widehat {\psi}$ and such that
\begin{equation}\label{eq:normalisation}
	\int_0^\infty \abs{\widehat{\psi}(t\xi)}^2 \frac{\dd t}{t}\thickapprox 1.
\end{equation}
Given an admissible weight $w$ and $\psi$ as above, let us define
\begin{equation}\label{eq:sigma}
	\sigma_w(\xi)=\brkt{\int_0^\infty \abs{\widehat{\psi}(t\xi)}^2 w^2(t)\frac{\dd t}{t}}^{\frac{1}{2}}.
\end{equation}
For simplicity of notation, we omit the explicit dependency of $\sigma_w$ on $\psi$.

\begin{lem}\label{lem:regularity_sigma} For any $\sigma_w$ defined as in \eqref{eq:sigma} one has that $\sigma_w,\sigma_w^{-1}\in \mathcal{C}^{\infty}(\R^n\setminus\{0\})$ and $\sigma_w^{-1}$ is bounded. Moreover, for any multiindex $\alpha$
\begin{enumerate}[itemindent=0pt,label=\rm\bfseries (\arabic*), fullwidth]
	\item\label{eq:one} $\abs{\partial^\alpha_\xi \brkt{\sigma_w(\xi)}}\lesssim w\brkt{1/\abs{\xi}} \abs{\xi}^{-\abs{\alpha}}$;
	\item\label{eq:two}\label{eq:sima_inv} $\abs{\partial^\alpha_\xi \brkt{1/\sigma_w(\xi)}}\lesssim w\brkt{1/\abs{\xi}}^{-1} \abs{\xi}^{-\abs{\alpha}}$.
\end{enumerate}
In particular, $1/\sigma_w$ is a H\"ormander-Mikhlin multiplier.
\end{lem}
\begin{proof}
Let $\Psi(\xi):=\abs{\widehat{\psi}(\xi)}^2\in \mathcal{C}^\infty_c(\R^n)$, which is supported in an annulus. Substituting this in \eqref{eq:sigma} and changing variables yield
\[
	\begin{split}
	\d^\alpha_\xi \sigma_w^2(\xi) %
=\abs{\xi}^{-\abs{\alpha}}\int_0^\infty \d^\alpha_\xi \Psi(t{\rm sgn}(\xi)) t^{\abs{\alpha}} w^2(t/\abs{\xi})\frac{\dd t}{t},
	\end{split}
\]
where ${\rm sgn}(\xi):=\xi/\abs{\xi}$. Since $\Psi$ is supported in an annulus, one has that
$t\thickapprox 1$ in the integral above and therefore Corollary \ref{cor:equivalence} and the boundedness of $\d^\alpha \Psi$ yield
\begin{equation}\label{eq:square}
	\abs{\d^\alpha_\xi \sigma_w^2(\xi) }\lesssim \abs{\xi}^{-\abs{\alpha}} w^2(1/\abs{\xi}).
\end{equation}
Observe that a change of variables, {Corollary \ref{cor:equivalence}} and \eqref{eq:normalisation} imply
\begin{equation}\label{eq:sigma equiv weight}
	\sigma_w(\xi)=\brkt{\int_0^\infty \Psi(t{\rm sgn}(\xi)) w^2(t/\abs{\xi})\frac{\dd t}{t}}^{1/2}\thickapprox w(1/\abs{\xi}).
\end{equation}
By Leibniz rule we have
\[
	\d^{\alpha} \sigma_w(\xi)=\frac{1}{2\sigma_w(\xi)}\brkt{\d^\alpha(\sigma_w^2(\xi))-\sum_{\overset{\alpha=\beta+\gamma}{ \beta\neq 0\neq \gamma}} c_{\beta,\gamma}\,\d^\beta \sigma_w(\xi) \d^{\gamma}\sigma_w(\xi)}.
\]
Therefore, \ref{eq:one} follows by using this and \eqref{eq:square}-\eqref{eq:sigma equiv weight}, followed by an inductive argument.

 To prove \ref{eq:two} we use \ref{eq:one}, \eqref{eq:sigma equiv weight} and apply the Leibniz rule to the function $\sigma_w(\xi)/\sigma_w(\xi)$. Observe also that \ref{A2} and \eqref{eq:sigma equiv weight} imply that $\sigma_w^{-1}$ is bounded, which yields that $\sigma_w^{-1}$ is a H\"ormander-Mikhlin multiplier. The verification of the details are left to the interested reader.
\end{proof}

\begin{prop}\label{prop:independency} Let $\psi_1$ and $\psi_2$ be as above, and let $\sigma_w^1$ and $\sigma_w^2$ be the corresponding symbols defined as in \eqref{eq:sigma}. Then
\[
	\norm{\brkt{\frac{\sigma_w^1}{\sigma_w^2}}(D) f}_{H^1}\thickapprox \norm{f}_{H^1},
 \]
 for any $f\in H^1(\R^n)$.
\end{prop}
\begin{proof}  Let us first observe that from Lemma \ref{lem:regularity_sigma} and the Leibniz formula, it follows that
 for any multiindex $\alpha$
 $$\abs{\partial^\alpha_\xi \brkt{\frac{\sigma_w^1(\xi)}{\sigma_w^2(\xi)}}}\lesssim  \abs{\xi}^{-\abs{\alpha}}.$$
In particular, this implies that ${\sigma_w^1}/{\sigma_w^2}$ is a H\"ormander-Mikhlin multiplier. Therefore, \cite{MR807149}*{Thm. III.7.30} yields that
\[
	\norm{\brkt{\frac{\sigma_w^1}{\sigma_w^2}}(D) f}_{H^1}\lesssim\norm{f}_{H^1},
\]
for any $f\in H^1(\R^n)$. A similar estimate also holds for the operator with symbol  ${\sigma_w^2}/{\sigma_w^1}$. Hence the result follows since
\[
	\norm{f}_{H^1}=\norm{\brkt{\frac{\sigma_w^2}{\sigma_w^1}}(D)\brkt{\brkt{\frac{\sigma_w^1}{\sigma_w^2}}(D) f}}_{H^1}\lesssim \norm{\brkt{\frac{\sigma_w^1}{\sigma_w^2}}(D) f}_{H^1}.
\]
\end{proof}

Let $Z(\R^n)$ be the the complete locally convex space defined by
\[
	Z(\R^n):=\set{f\in \mathcal{S}(\R^n):\, \partial^\alpha \widehat{f}(0)=0 \quad \mbox{for every multi-index $\alpha$}},
\]
as a topological subspace of $\mathcal{S}(\R^n)$. Let $Z'(\R^n)$ be its topological dual, which can be identified as the factor space $\mathcal{S}'(\R^n)/\mathcal{P}(\R^n)$, where $\mathcal{P}(\R^n)$ is the space of polynomials in $\R^n$. That is, $Z'(\R^n)$ consists of the class of distributions of the form $f+P$ where $f\in \mathcal{S}'(\R^n)$ and $P\in \mathcal{P}(\R^n)$.

\begin{prop} Let $m(\xi)$ be either $\sigma_w$ or $1/\sigma_w$. Then the functional given by
\begin{equation*}\label{eq:multiplier}
	I_{m}f(x)=m(D)f
\end{equation*}
defines a continuous linear operator on $Z(\R^n)$. Moreover
\[
	I_ {\sigma_w}I_ {1/\sigma_w}f=I_ {1/\sigma_w}I_ {\sigma_w}f=f,\quad \mbox{for any $f\in Z(\R^n)$}.
\]
\end{prop}
\begin{proof}
We will only show the result in the case $m=\sigma_w$, since the case of $m=1/\sigma_w$ is done similarly. It suffices to show that for any $f\in Z(\R^n)$, $F(\xi)=m(\xi)\widehat{f}(\xi) \in \mathcal{S}(\R^n)$ and $\partial^\alpha F(0)=0$ for any multiindex $\alpha$. To this end, observe first that trivially $F\in\mathcal{C}^{\infty}(\R^n\setminus\{0\})$. So, it remains to study what happens at the origin. Since $f\in Z(\R^n)$,
for any $L>0$ and any $\alpha$, $\lim_{\xi \to 0} |{\partial^\alpha\widehat{f}(\xi)}|{\abs{\xi}^{-L}}=0$. Moreover, for any $\xi\neq 0$,  Leibniz's rule, Lemma \ref{lem:regularity_sigma} and \ref{A3} yield
\begin{equation}\label{eq:F}
	\begin{split}
	|{\partial^\alpha F(\xi)}| &\lesssim w(1/\abs{\xi}) \sum\abs{\xi}^{-\abs{\alpha_1}} \abs{\partial^{\alpha_2} \widehat{f}(\xi)}
	\\
	&\lesssim (1+\abs{\xi})^{N} \sum\abs{\xi}^{-\abs{\alpha_1}} \abs{\partial^{\alpha_2} \widehat{f}(\xi)} ,
	\end{split}
\end{equation}
where the sum runs over all the multiindices $\alpha_1+\alpha_2=\alpha$.
This yields $\lim_{\xi\to 0} |{\partial^\alpha F(\xi)}|=0$. Hence, for any $\alpha$, $\partial^\alpha F(\xi)$ can be extend continuously to $0$ by setting $\partial^\alpha F(0):=0$.
Now \eqref{eq:F} yields $\sup_{\abs{\xi}\geq 1}|{\xi^\beta \partial^\alpha F(\xi)}|<+\infty.$\\
On the other hand, the fact that $\sup_{\abs{\xi}\leq 1}|\xi^\beta \partial^\alpha F(\xi)|$ is finite is a consequence of continuity and the compactness of the unit ball.

The last assertion is a direct consequence of the definition of the operators.
\end{proof}

\begin{cor}\label{cor:technic_1} Let $m(\xi)$ be either $\sigma_w$ or $1/\sigma_w$. The functional $I_{m}(u)$ defined by
\[
	I_{m}(u)(f):=u(I_{m} f),\quad \mbox{for any $u\in Z'(\R^n)$, $f\in Z(\R^n)$},
\]
is a continuous, linear, one-to-one mapping of $Z'(\R^n)$ onto itself. Moreover,
\[
	I_ {\sigma_w}I_ {1/\sigma_w}u=I_ {1/\sigma_w}I_ {\sigma_w}u=u,\quad \mbox{for any $u\in Z'(\R^n)$}.
\]
\end{cor}

\begin{prop}\label{prop:boundedness_H1_BMO} The operator $I_ {1/\sigma_w}$ maps $H^1(\R^n)\to H^1(\R^n)$ and $\mathrm{BMO}(\R^n)\to \mathrm{BMO}(\R^n)$ continuously, and
${\rm Ker }_{H^1}\, I_ {1/\sigma_w}\equiv\{0\}.$
\end{prop}
\begin{proof} The boundedness of the operator $I_ {1/\sigma_w}$ on $H^1$ is a consequence of the boundedness of H\"ormander-Mikhlin multipliers on these spaces. More specifically, the boundedness on $H^1$ can be found in \cite{MR807149}*{Thm. III.7.30}, and using that together with the self-adjointness of $I_ {1/\sigma_w}$ and duality, we obtain the boundedness on BMO.

On the other hand,  since $I_ {1/\sigma_w}$ is bijective in $Z'(\R^n)$ and $H^1(\R^n)\subset Z'(\R^n)$, Corollary \ref{cor:technic_1} and the characterisation of $Z'(\R^n)$ yield that if for $f\in H^1(\R^n)$, ${ I_ {1/\sigma_w} f}=0,$ then $f=0$ (modulo polynomials) i.e. $f=P$ where $P$ is a polynomial. Now since $H^1\subset L^1$, it follows that $P=0$ and therefore  $I_ {1/\sigma_w}$ restricted to $H^1(\R^n)$ is injective.
\end{proof}

The previous proposition allows us to define the following Hardy-Sobolev-type space, $H^1_{1/\sigma_w}(\R^n)$.
\begin{defn}\label{defn:hardy_sobolev} We shall denote by $H^1_{1/\sigma_w}(\R^n)$ the space
\[
	H^1_{1/\sigma_w}(\R^n):=\set{f\in H^1(\R^n):\, f=I_{1/\sigma_w} h,\, h\in H^1(\R^n)},
\]
endowed with the norm
\[
	\norm{f}_{H^1_{1/\sigma_w}(\R^n)}=\norm{h}_{H^1(\R^n)}.
\]
\end{defn}

\begin{prop} The space $(H^1_{1/\sigma_w},\norm{\cdot}_{H^1_{1/\sigma_w}})$ is a Banach space.
\end{prop}
\begin{proof} It is easy to see that $\norm{\cdot}_{H^1_{1/\sigma_w}}$ is a norm. For proving the completeness one observes that for any Cauchy-sequence $\{f_n\}_n\subset H^1_{1/\sigma_w}$, there exists a Cauchy-sequence $\{h_n\}_n \subset H^1(\R^n)$ such that, for any $n\geq 1$, $I_{1/\sigma_w} h_n=f_n$. By completeness of $H^1(\R^n)$, there exists $h\in H^1(\R^n)$ such that $\lim_n h_n=h$. Then, defining $f=I_{\sigma_w} h\in H^1_{1/\sigma_w}$ and, by the linearity of $I_{1/\sigma_w}$ one has $\lim_n \norm{f_n-f}_{H^1_{1/\sigma_w}}=0$.
\end{proof}

The previous result and the fact that $I_ {1/\sigma_w}$ is an isomorphism in $Z(\R^n)$ implies that
$$
	Z(\R^n)\subset H^1_{1/\sigma_w}(\R^n)\subset H^1(\R^n).
$$

 In the following proposition, we show that the space $H^1_{1/\sigma_w}(\R^n)$ is well-defined. By this we mean that the definition \ref{defn:hardy_sobolev} above depends only on $w$ and not on the underlying function $\psi$, in the sense that different choices of $\psi$ induce equivalent norms.
\begin{prop}
Let $\psi_1$ and $\psi_2$ be as above, and let $\sigma_w^1$ and $\sigma_w^2$ be the corresponding symbols defined as in \eqref{eq:sigma}. Then $H^1_{1/\sigma_w^1}(\R^n)=H^1_{1/\sigma_w^2}(\R^n)$ and for any $f\in H^1_{1/\sigma_w^1}(\R^n)$
\[
	\norm{f}_{H^1_{1/\sigma_w^1}(\R^n)}\thickapprox
	\norm{f}_{H^1_{1/\sigma_w^2}(\R^n)},
 \]
 with constants depending on $\psi_1$ and $\psi_2$.
\end{prop}
\begin{proof} Let $f\in H^1_{1/\sigma_w^1}(\R^n)$ and let $h\in H^1(\R^n)$ such that  $f=I_{1/\sigma_w^1} h$. Then
\[
	f=I_{1/\sigma_w^2} (I_{\sigma_w^2/\sigma_w^1} h),
\]
 and by  Definition \ref{defn:hardy_sobolev} and Proposition
\ref{prop:independency} one has
 \[
 	\norm{f}_{H^1_{1/\sigma_w^2}(\R^n)}=\norm{I_{\sigma_w^2/\sigma_w^1} h}_{H^1(\R^n)}\thickapprox \norm{h}_{H^1(\R^n)}=
 	\norm{f}_{H^1_{1/\sigma_w^1}(\R^n)}.
 \]
\end{proof}
From the definition of the space $H^1_{1/\sigma_w}(\R^n)$, it follows that $I_{1/\sigma_w}$ is an isometric isomorphism between $H^1(\R^n)$ and $H^{1}_{1/\sigma_w}(\R^n)$. The inverse of $I_{1/\sigma_w}$ is obtained by its restriction to $H^{1}_{1/\sigma_w}(\R^n)$ and is denoted by $S_{\sigma_w}$. This implies that the corresponding dual spaces are isomorphic. More precisely, for any $\Lambda\in (H^1_{1/\sigma_w}(\R^n))^*$ there exists a unique (modulo constants) $b_\Lambda \in \BMO$ given by  $b_\Lambda=I_{1/\sigma_w}\Lambda$, such that for any $h\in H^1(\R^n)$
 \[
 	 \esc{\Lambda, I_{1/\sigma_w} h}=\esc{b_\Lambda,h},
 \]
 and conversely, for any $b\in \BMO$ there exists $S_{\sigma_w}(b):=\Lambda_b\in  (H^1_{1/\sigma_w}(\R^n))^*$ such that
  \[
 	 \esc{\Lambda_b, I_{1/\sigma_w} h}=\esc{b,h}.
 \]
 This motivates the following definition:
 \begin{defn}\label{defn: BMO_w} We define the space $\BMO_{\sigma_w}$ as
\begin{equation*}
	\BMO_{{\sigma_w}}:=\set{\Lambda\in \mathcal{S}'(\R^n):\quad \norm{I_{1/\sigma_w}\Lambda}_{\BMO}<+\infty},
\end{equation*}
and equip it with the norm
\[
	\norm{\Lambda}_{\BMO_{{\sigma_w}}}:=
	\norm{I_{1/\sigma_w}\Lambda}_{\BMO}.
\]
\end{defn}
As such, one has that $(H^1_{1/\sigma_w}(\R^n))^*=\BMO_{\sigma_w}$. { Moreover, Proposition \ref{prop:boundedness_H1_BMO} yields that $\BMO\subset \BMO_{\sigma_w}$. }

\section{Estimates for localisation operators }\label{local estimates}
Let $\phi, \psi$ be Schwartz class functions, with spectrum included  in a ball and an annulus around $0$ respectively, and $\int \phi(x)\dd x\neq 0$. For any $0<t<\infty$, one defines the frequency localisation operators $P_t$ and $Q_t$ as
\begin{equation}\label{eq:pt_qt}
	 P_t f=\widehat{\phi}(tD) f, \qquad Q_t f=\widehat{\psi}(tD) f.
\end{equation}
In what follows, we shall make use of the following proposition.
\begin{prop}%
There exists a constant $C$ such that for any $f\in \BMO$ and any $t>0$
\begin{equation}\label{eq:Q1}
	\norm{Q_t f}_{L^\infty}\leq C \norm{f}_{\BMO}.
\end{equation}
Moreover, for any $f\in \bmo$
\begin{equation}\label{eq:Pt}
	\norm{P_t f}_{L^\infty}\lesssim \brkt{1+\log_+ \frac{1}{t}} \norm{f}_{\bmo}.
\end{equation}
 \end{prop}
 
\begin{proof}
The statement \eqref{eq:Q1} is a classical result and could be found in \cite{S}*{p. 161}. %
So it remains to prove \eqref{eq:Pt}. Without loss of generality we can assume that $\widehat{\phi}$ is equal to one in a neighbourhood of the origin. If not, we can find a compactly supported smooth function function $\widehat{\varphi}$ such that is equal to one on the support of $\widehat{\phi}$. In this way we can write $P_t f=P_t \tilde{P_t} f$, where $\tilde{P_t} f=\widehat{\varphi}(tD) f$, which yields
\[
	\norm{P_t \tilde{P_t} f}_{L^\infty}\leq \int \abs{\phi(x)}\dd x \norm{\tilde{P_t} f}_{L^\infty}.
\]
So, it would be sufficient to prove the result for $\tilde{P_t} f$.

   Observe that for any $0<t<1$
\[
	P_1f(x)-P_tf(x)=\int_t^1 \int  \nabla_\xi \widehat{\phi}(s\xi)\cdot (s\xi) \widehat{f}(\xi)\, e^{ix\cdot \xi}\, \ddd \xi \frac{\dd s}{s}.
\]
If we now set  $\widehat{\Psi}(\xi)=\nabla_\xi\widehat{\phi}(\xi)\cdot \xi$, then  $\widehat{\Psi}$ is a smooth function supported in an annulus around the origin, because $\widehat{\phi}$ is compactly supported { and equal to one in a neighbourhood of the origin}. Then we have that
\[
	P_1f(x)-P_tf(x)=\int_t^1 \widehat{\Psi}(sD) f(x)\frac{\dd s}{s}.
\]
Therefore, using \eqref{eq:bmo} and \eqref{eq:Q1} we obtain
\[
	\norm{P_t f}_{L^\infty}\leq \norm{P_1 f}_{L^\infty}-\log t \norm{f}_{\BMO}\leq (1-\log t)\norm{f}_{\bmo}.
\]
On the other hand, for $t\geq 1$, $\widehat{\phi}(t\xi)$ is supported in a fixed ball independent of $t$. So, we can find an smooth compactly supported  function $\widehat{\Theta}(\xi)$ equal to one in a neighbourhood of the origin such that $P_t f=P_t(\widehat{\Theta}(D)f)$. Thus Minkowskii's inequality and \eqref{eq:bmo} yield
\[
	\norm{P_t f}_{L^\infty}\lesssim \norm{\widehat{\Theta}(D)f}_{L^\infty}\lesssim \norm{f}_{\bmo}.
\]
\end{proof}

\begin{defn}\label{defn: X_w}
For any admissible weight $w$ we shall denote by $X_w$ the space of functions $f\in \BMO$ such that
\begin{equation}\label{eq:X_w}
	\norm{f}_{X_w}:=\norm{f}_{\BMO}+\sup_{t>0} \frac{\norm{P_t f}_{L^\infty}}{w(t)}<+\infty.
\end{equation}
\end{defn}

Note that $X_w$ is nontrivial as it clearly contains $L^\infty$. Also $X_w\subset \BMO$. Moreover, since
\begin{equation}\label{eq:lower_bound}
	\norm{f}_{X_w}\geq \norm{f}_\BMO+\frac{\norm{P_1 f}_{L^\infty}}{w(1)}\thickapprox \norm{f}_{\bmo},
\end{equation}
it follows from \eqref{eq:bmo} that $X_w$ is embedded in  $\bmo$.

Here we shall also observe that for any $f\in \BMO$,
\begin{equation} \label{eq:averages}
	\sup_t \norm{P_tf- c{ \rm Avg}_{B_t(\cdot)} f}_{L^\infty}\lesssim \norm{f}_{\BMO},
\end{equation}
where ${ \rm Avg}_{B_t(x)} f=\frac{1}{\abs{B_t(x)}}\int_{B_t(x)} f(y)\dd y$, $B_t(x)$ is the ball centered at $x$ of radius $t$ and $c=\int \phi(x)\dd x$. The proof of \eqref{eq:averages} is similar to that of \cite{G}*{Prop. 7.1.5.ii}. An easy consequence of this and the fact that $w(t)\geq 1$ is that
\[
	\norm{f}_{X_w}\thickapprox \norm{f}_{\BMO}+\sup_{t>0} \frac{\norm{{ \rm Avg}_{B_t(\cdot)} f}_{L^\infty}}{w(t)}<+\infty.
\]
In this way, we see that the definition of the class $X_w$ does not depend on the different choices of function $\phi$ associated to $P_t$, in the sense that different choices of $\phi$ induce equivalent norms.
Moreover, this expression allows us to rewrite
\begin{equation*} %
	\sup_{t>0} \frac{\norm{{ \rm Avg}_{B_t(\cdot)} f}_{L^\infty}}{w(t)}=\sup  \frac{1}{w(r(B))) \abs{B}}\abs{\int_{B} f(x)\dd x},
\end{equation*}
where $r(B)$ is the radius of the ball $B$ and the supremum is taken over the family of all euclidean balls in $\R^n$.
Let us recall that  for $\eta\colon [0,\infty) \rightarrow
[0,\infty)$, the Morrey space $\mathcal{M}^{1,\eta}$ is the class of all locally integrable functions $f$, for which
$$
	\left\|f\right\|_{1,\eta} = \sup \frac{\eta(r(B))}{|B|}\int_{B}\left|f(x)\right|\,\dd x<+\infty,
$$
see  e.g. \cite{MR2833565} and the references therein. Therefore, we have just proved that the class $X_w$ contains the space $\mathcal{M}^{1,\frac{1}{w}}\cap \BMO$.

As a consequence of \eqref{eq:Pt}, \eqref{eq:X_w}, \eqref{eq:lower_bound} we have the following alternative  characterisations of $\bmo$:
\begin{thm}\label{cor:bmo} For $w=1+\log_+ \frac{1}{t}$, we have $X_w=\bmo$ and
\[
	\begin{split}
	\norm{f}_\bmo &\thickapprox \norm{f}_\BMO+ \sup_{t>0} \frac{\norm{P_t f}_{L^\infty}}{1+\log_+ \frac{1}{t}}\thickapprox\norm{f}_{\BMO}+\sup  \frac{1}{w(r(B))) \abs{B}}\abs{\int_{B} f(x)\dd x}.
	\end{split}
\]
\end{thm}

\begin{rem}
For $w(t)=1+\log_+1/t$, we have $\mathcal{M}^{1,\frac{1}{w}}\cap \BMO \subset \bmo$ as a consequence of Corollary \ref{cor:bmo}.
\end{rem}
\section{Boundedness of bilinear Paraproducts}\label{paraproducts}
\subsection{Boundedness of constant coefficient paraproducts}

\begin{defn} Let $P_t$ and $Q_t$ be defined as in \eqref{eq:pt_qt}. A bilinear (constant coefficient) paraproduct $\Pi$ on $\R^d$ is a bilinear operator of the form
$$ \Pi(f,g)(x):= \int_0^\infty Q_t f(x)\, P_tg(x) m(t)\frac{\dd t}{t},$$
where $m$ is a bounded measurable function.
\end{defn}

 These paraproducts are of particular interest when dealing with bilinear Coifman-Meyer multipliers (see Section \ref{kato-ponce}).

\begin{thm}\label{thm:main} For any admissible weight $w$ we have
\[
	\norm{\Pi (f,g)}_{\BMO_{\sigma_w}} \leq C \norm{f}_{\BMO}\norm{g}_{X_w}.
\]
\end{thm}

Before proceeding with the proof, we note that as a consequence of  Theorem \ref{thm:main} we have the following result:

\begin{thm}\label{thm:main_thm} For any admissible weight $w$
\[
	\norm{\Pi (f,g)}_{\BMO_{\sigma_w}} \leq C \norm{f}_{\BMO}\norm{g}_{\mathcal{M}^{1,\frac{1}{w}}\cap \BMO}.
\]
For $w=1$,
\[
	\norm{\Pi (f,g)}_{\BMO} \leq C \norm{f}_{\BMO}\norm{g}_{L^\infty(\R^n)}.
\]
For $w=1+\log_+ \frac{1}{t}$,
\[
	\norm{\Pi (f,g)}_{\BMO_{\sigma_w}} \leq C \norm{f}_{\BMO}\norm{g}_{\bmo}.	
\]
\end{thm}
\begin{proof}
The strategy consists of applying Theorem \ref{thm:main} to the various weights under consideration.

The first assertion is a direct consequence of the embedding of $\mathcal{M}^{1,\frac{1}{w}}\cap \BMO$ into $X_w$ observed above.

For the second claim, one observes that $w=1$ yields $X_w=L^\infty(\R^n)$ and $\BMO_{\sigma_w}=\BMO$ with equivalent norms in both cases.

For the third claim, Corollary \ref{cor:bmo} implies that $\bmo=X_w$. This concludes the proof of the theorem.
\end{proof}

The proof of Theorem \ref{thm:main} requires a couple of technical results.  The first one is due to Carleson \cite{Carleson} (see also  \cite{G}*{Theorem 7.3.8}). 
\begin{thm}\label{tmh:Carleson_thm}
\begin{enumerate}[itemindent=0pt, label=\bfseries (C\arabic*),fullwidth]
\item \label{Carleson:1}For any $b\in \BMO$, the measure $\dd \mu(x,t)$ given by
\[
	\dd \mu(x,t)=\abs{Q_t b(x)}^2\dd x\frac{\dd t}{t},
\]
is a Carleson measure with norm a constant times $\norm{b}_{\BMO}^2$.

\item \label{Carleson:2} Let $\delta,A>0$. Suppose that $\{K_t\}_{t>0}$ are functions in $\R^n\times \R^n$ that satisfy
\[
	\abs{K_t(x,y)}\leq {A t^{-n}}{\brkt{1+\frac{\abs{x-y}}{t}}^{-n-\delta}},
\]
for all $t>0$ and $x,y\in  \R^n$. Let $R_t$ be the linear operator
\[
	R_tf(x)=\int K_t(x,y) f(y)\dd y.
\]
Suppose that $R_t(1)=0$ for all $t>0$ and that there exists a constant $B>0$ such that
\[
	\int_0^\infty \int \abs{R_t f(x)}^2 \dd x \frac{\dd t}{t}\leq B^2 \norm{f}_{L^2(\R^n)}^2,
\]
for all $f\in L^2(\R^n)$, then for any $b\in \BMO$, the measure $\dd \mu(x,t)$ given by
\[
	\dd \mu(x,t)=\abs{R_t b(x)}^2\dd x\frac{\dd t}{t},
\]
is a Carleson measure with norm a constant times $(A+B)^2\norm{b}_{\BMO}^2$.
\end{enumerate}
\end{thm}
 
We will also need the following proposition, whose proof can be found in \cite{RRS2}*{Proposition 4.11}.

\begin{prop}\label{prop:monster}
Let $F\in H^1$, $v\in L^\infty_{t,x}$, and $T_{t}$ be the convolution operator given by
\[
	T_{t} (f)(x)=\int f(x-y)\dd \nu_t(y),
\]
with $\{\lambda_t\}_t$  being finite measures such that for some $\delta>0$ and for any $t>0$,
\begin{equation*} %
	 \int \brkt{1+\frac{\abs{x-y}}{t}}^{-n-\delta} \dd \abs{\lambda_t}(y)\lesssim \brkt{1+\frac{\abs{x}}{t}}^{-n-\delta}.
\end{equation*}
Let $G(t,x)$ be a measurable function on $\R^{n+1}_+$ such that $$\dd \mu_G(t,x)=\abs{G(t,x)}^2\frac{\dd t}{t}\dd x$$ is a Carleson measure with Carleson norm $\norm{\dd \mu_G}_{\mathcal{C}}$. Then
\[
	\abs{\int\int_0^\infty Q_t T_t F(x)\, G(t,x) v(t,x)\frac{\dd t}{t}\dd x}\leq C \norm{F}_{H^1}\norm{\dd \mu_G}_{\mathcal{C}}^{\frac{1}{2}}\norm{v}_{L^\infty_{t,x}}.
\]
\end{prop}
\subsection*{\bf Proof of Theorem \ref{thm:main}}

Let $w$ be an admissible weight.  Suppose that $\widehat{\psi}$ is supported in $0<\alpha\leq \abs{\xi}\leq \beta$ and $\widehat{\psi}$ has its support in $\abs{\xi}\leq \beta$.  Then we can decompose $\phi=\psi_1+\phi_1$ where $\widehat{\psi_1}$ vanishes in $\abs{\xi}\leq \alpha/8$ and $\widehat{\phi_1}$ is supported in $\abs{\xi}\leq \alpha/4$. Then one can  find a smooth real valued radial function $\psi_2$ with spectrum included in an annulus and equal to one in $\alpha/4\leq \abs{\xi}\leq 2\beta$, and another smooth function $\phi_2$ with $\widehat{\phi_2}$ compactly supported and equal to one in $\abs{\xi}\leq 2\beta$ such that we can decompose $\Pi(F,G)$ as
\[
	\Pi(F,G)=\Pi_1(F,G)+\Pi_2(F,G)
\]
where
 \[
 	\Pi_1(F,G)=\int_0^\infty Q_t^{(2)}\brkt{\brkt{Q_t F} \brkt{P_t^{(1)} G}}  m(t)\frac{\dd t}{t},
 \]
 and
  \[
 	\Pi_2(F,G)= \int_0^\infty P_t^{(2)} \brkt{\brkt{Q_t F} \brkt{Q_t^{(1)} G}} m(t)\frac{\dd t}{t},
 \]
 and $P_t^{(j)}$ and $Q_t^{(j)}$ are associated to $\widehat{\phi_j}$ and $\widehat{\psi_j}$ respectively.

 \subsection*{Estimates for $\Pi_2$}
 We shall prove that $\Pi_2: \BMO\times \BMO \to \BMO$. The Cauchy-Schwarz inequality and \ref{Carleson:1} in Theorem \ref{tmh:Carleson_thm} imply that 
 $\abs{Q_t F Q_t^{(1)}G }\dd x\frac{\dd t}{t}$ defines a Carleson measure with Carleson norm bounded by a constant multiple of $\norm{F}_\BMO\norm{G}_\BMO$.

For any $H\in H^1$ we can write
\begin{equation}\label{eq:first_term}
	\left<\Pi_2(F,G),H\right>=\int \int_0^\infty {\brkt{Q_t F} \brkt{Q_t^{(1)} G}} \big(P_t^{(2)} H\big)\, m(t)\frac{\dd t}{t}\dd x,
\end{equation}
where we have implicitly used that $P_t^{(2)}(m(t)H)=\big(P_t^{(2)} H\big)\, m(t).$

Using Fefferman-Stein's result \cite{FS} and \eqref{eq:H1} we obtain
\[
	\begin{split}
	\abs{\left<\Pi_2(F,G),H\right>} &\lesssim \norm{F}_{\BMO}\norm{G}_{\BMO} \norm{m}_{L^\infty} \int \sup_{t>0}\sup_{\abs{x-y}<t} \abs{P_t^{(2)} H(y)}\dd x\\
&\lesssim \norm{m}_{L^\infty} \norm{H}_{H^1}  \norm{F}_{\BMO}\norm{G}_{\BMO}.
	\end{split}
\]
   Therefore, we obtain that for any $F,G\in \BMO$ and $H\in H^1(\R^n)$
 \begin{equation}\label{eq:pi_6}
 	\abs{\left<\Pi_2(F,G),H\right>}\lesssim \norm{m}_{L^\infty} \norm{F}_{\BMO}\norm{G}_{\BMO}\norm{H}_{H^1}.
 \end{equation}
\subsection*{Estimates for $\Pi_1$}

For the admissible weight $w$, if we define $v(t,x)=P_t^{(2)}(G)(x)m(t)/w(t)$, by \eqref{eq:X_w} we have
 that
 \begin{equation}\label{eq:v}
 	\norm{v(t,\cdot)}_{L^\infty}\lesssim \norm{m}_{L^\infty}\norm{G}_{X_w}
\end{equation}
and
 \[
 	\left<\Pi_1(F,G),H\right>=\int\int_0^\infty \brkt{Q_t F} \big({Q_t^{(2)}H }\big)w(t) v(t,x)\frac{\dd t}{t}\dd x,
 \]
where once again we have used the fact that $Q_t^{(2)}(m(t)H)=\big(Q_t^{(2)} H\big)\, m(t).$

Thus, taking $H=I_{{1}/{\sigma_w}} h$ with $h\in H^1(\R^n)$ and $\widehat{\psi_3}$ a smooth annulus supported function such that it is equal to one on the support of $\widehat{\psi}_2$,  one can write 
\[
	w(t)Q_t^{(2)}H=R_t(Q_t^{(3)} h),
\]
where 
\[
	R_tf(x)=\int K_t(x-y) f(y)\dd y,
\]
with
\begin{equation*}
		K_t(z)= w(t)\int \frac{\widehat{\psi}_2(t\xi)}{\sigma_w(\xi)} e^{iz\cdot\xi}\ddd \xi.
\end{equation*}
We claim that for any $N\geq 1$
\begin{equation} \label{eq:Kernel}
	\abs{K_t(z)}\leq A{t^{-n}}{\brkt{1+\frac{\abs{z}}{t}}^{-2N}}.
\end{equation}
Indeed, integration by parts yields that for any $N\geq 1$
\[
	\begin{split}
	K_t(z)&=t^{-n} w(t)\int \frac{\widehat{\psi}_2(\xi)}{\sigma_w(\xi/t)} e^{i t^{-1}z\cdot\xi}\ddd \xi\\
	&=t^{-n} w(t) \brkt{\frac{\abs{z}}{t}}^{-2N}\int (-\Delta)^{N}\left[\frac{\widehat{\psi}_2(\xi)}{\sigma_w(\xi/t)}\right]e^{i t^{-1}z\cdot\xi}\ddd \xi,
	\end{split}
\]
which implies the estimate above since the Leibniz rule, the fact that $\abs{\xi}\approx 1$ and  \ref{eq:sima_inv} in Lemma \ref{lem:regularity_sigma} yield that
\[
	\abs{(-\Delta)^{N}\left[\frac{\widehat{\psi}_2(\xi)}{\sigma_w(\xi/t)}\right]}\lesssim w(t)^ {-1}.
\]

 Thus we can write
 \begin{equation}\label{eq:estimate_pi4}
 	\begin{split}
 		\left<\Pi_1(F,G),H\right>&=\int\int_0^\infty  \brkt{Q_t F} \brkt{R_t {Q_t^{(3)}} h} v(t,x)\frac{\dd t}{t}\dd x\\
 		&=\int\int_0^\infty  \brkt{{Q_t^{(3)}} h} R_t^*\left[{\brkt{Q_t F} v(t,\cdot)}\right]\frac{\dd t}{t}\dd x.
 		\end{split}
 \end{equation}
 Now observe that
 \[
 	R_t^*\brkt{\brkt{Q_t F} v(t,\cdot)}(x)=\int J_t(x,z) F(z)\dd z,
 \]
 where
 \[
 	J_t(x,z)=\int K_t(x-y)v(t,y) \psi_t(y-z)\dd y.
 \]
 Since $\psi$ is a Schwartz function, \eqref{eq:v} holds and $K_t$ satisfies \eqref{eq:Kernel}, it can be shown that for any $N\geq 1$
 \[
 	\abs{J_t(x,z)}\lesssim  t^{-n}\brkt{1+\frac{\abs{x-z}}{t}}^{-2N}\norm{G}_{X_w}\norm{m}_{L^\infty}.
 \]
 On the other hand, for $F=1$, $R_t^*\brkt{\brkt{Q_t F} v(t,\cdot)}(x)=0$. Moreover, Plancherel's theorem, Corollary \ref{cor:equivalence}, and the Minkowski and Young inequalities yield
 \[
 	\begin{split}
 	\int\int_0^\infty &\abs{R_t^*\brkt{\brkt{Q_t F} v(t,\cdot)}(x)}^2 \frac{\dd t}{t}\dd x\thickapprox
 	\int\int_0^\infty \abs{Q_t^{(2)}\brkt{(Q_t F) v(t,\cdot)}(x)}^2 \frac{\dd t}{t}\dd x\\
 	&\leq \int\int_0^\infty \norm{\psi_{2,t}}_{L^1}^2\norm{v(t,\cdot)}_{L^\infty}^2 \abs{Q_t F(x)}^2
 	 \frac{\dd t}{t}\dd x\lesssim \norm{F}_{L^2}^2 \norm{G}_{X_w}^2\norm{m}_{L^\infty}^2.
 		\end{split}
 \]
Finally \ref{Carleson:2} in Theorem \ref{tmh:Carleson_thm} yields that for any $F\in \BMO$
 \[
 	\abs{R_t^*\brkt{\brkt{Q_t F} v(t,\cdot)}(x)}^2 \frac{\dd t}{t}\dd x
 \]
 defines a Carleson measure with Carleson norm bounded by a constant multiple of $$\norm{G}_{X_w}^2\norm{m}_{L^\infty}^2\norm{F}_{\BMO}^2.$$

Now applying Proposition \ref{prop:monster} to the last term in \eqref{eq:estimate_pi4}, with
 $G(t,x)=R_t^*\left[{\brkt{Q_t F} v(t,\cdot)}\right]$, $T_t f=f$ (that is, taking $\lambda_t$ to be the Dirac measure supported at the origin), and $v(t,x)=1$ yield
 \begin{equation}\label{eq:pi_4}
 	\abs{\left<\Pi_1(F,G),I_{\frac{1}{\sigma_w}}h\right>}\lesssim \norm{h}_{H^1}\norm{F}_{\BMO}\norm{G}_{X_w}\norm{m}_{L^\infty}.
 	 \end{equation}
To finish the proof of Theorem \ref{thm:main}, we put \eqref{eq:pi_6} and \eqref{eq:pi_4} together, and use  Proposition \ref{prop:boundedness_H1_BMO} which yields the $H^1$ boundedness of $I_{\frac{1}{\sigma_w}}$. Hence
 \[
 	\begin{split}
 	\Big|\left<\Pi(F,G), I_{\frac{1}{\sigma_w}}h\right>\Big|&\lesssim \norm{F}_{\BMO}\norm{G}_{X_w}\norm{h}_{H^1},
 	\end{split}
 \]
 which implies
 \[
 	\norm{ \Pi(F,G)}_{\BMO_{\sigma_w}}\lesssim \norm{F}_{\BMO}\norm{G}_{X_w},
 \]
 for any $F\in \BMO$, $G\in X_w$. This concludes the proof of Theorem \ref{thm:main}.

 \subsection{Boundedness of variable coefficient paraproducts}\label{xdependent}
Here we shall investigate the boundedness of certain variable coefficient paraproducts defined by
 $$\PP(f,g)(x):= \int_0^1 Q_t f(x)\, P_tg(x) m(t,x)\frac{\dd t}{t}.$$
 These paraproducts are of particular interest when dealing with bilinear pseudodifferential and Fourier integral operators (see Section \ref{FIOs}). Our main result concerning $\PP(f,g)$ is as follows:
\begin{thm}\label{xdependent version}
 Let $\PP$ be the operator given as above, where $m$ is a bounded measurable function such that
\[
	 \sup_t \norm{m(t,\cdot)}_{L^\infty}+\sup_t \norm{\nabla_x m(t,\cdot)}_{L^\infty}<+\infty.
\]
Let $w$ be an admissible weight such that $\int_0^1 w(t)\dd t<+\infty$. Then
\[
	\norm{\PP(f,g)}_{\BMO_{\sigma_w}} \leq C \norm{f}_{\bmo}\norm{g}_{X_w}.
\]
\end{thm}
 \begin{proof} We shall start by observing that, since $t\in (0,1)$, there exist a Schwartz function $\Theta$ such that $\widehat{\Theta}$ is supported in a small neighbourhood of the origin such that for any $F$,
 \[
 	Q_t\brkt{F}=Q_t\brkt{(1-\widehat{\Theta})(D)F}.
 \]
 In this way, we observe that
 \begin{equation}\label{eq:obs}
 	\PP(F,G)(x)= \PP((1-\widehat{\Theta})(D)F, G).
\end{equation}

For proving the result, one proceeds as in the proof of Theorem \ref{thm:main_thm} and decomposes $\PP(F,G)=\PP_1(F,G)+\PP_2(F,G)$, and pair them with a function $H=I_{1/\sigma_w} h\in H^1_{1/\sigma_w}(\R^n)$.
 Then
 \[
	\left<\PP_2(F,G),H\right>=\int \int_0^1{\brkt{Q_t F} \brkt{Q_t^{(1)} G}} \big(P_t^{(2)} (M_m H)\big)\frac{\dd t}{t}\dd x,
\]
where $M_m(H)(x)= m(t,x) H(x)$. Now if we introduce the operator
\[
	[P_t^{(2)},M_m] H(x)=P_t^{(2)} (M_m H)(x)-m(t,x)P_t^{(2)} (H)(x),
\]
we can write
 \[
 	\begin{split}
	\left<\PP_2(F,G),H\right> &=\int \int_0^1 {\brkt{Q_t F} \brkt{Q_t^{(1)} G}} [P_t^{(2)},M_m] H(x)\frac{\dd t}{t}\dd x+\\
	&+\int \int_0^1 {\brkt{Q_t F} \brkt{Q_t^{(1)} G}}\big(P_t^{(2)} H\big) m(t,\cdot)\frac{\dd t}{t}\dd x = I+II.
	\end{split}
\]
The second term can be dealt with in the same way as \eqref{eq:first_term}, which yields
\[
	\abs{II}\lesssim \sup_t \norm{m(t,\cdot)}_{L^\infty} \int \sup_{0<t<1}\sup_{\abs{x-y}<t} \abs{P_t^{(2)} H(y)}\dd x\lesssim \sup_t \norm{m(t,\cdot)}_{L^\infty} \norm{H}_{H^1}.
\]
For the first term we observe that $[P_t^{(2)},M_m] H(x)(x) =\int K(t,x,y) H(y)\dd y$, with
 \[
 	K(t,x,y)=t^{-n}\phi_2\brkt{\frac{x-y}{t}}\big(m(t,y) - m(t,x)\big)
\]
Observe that by the mean-value theorem, we have that
\[
|m(t,y) - m(t,x)| \lesssim\sup_t \norm{\nabla_x m(t,\cdot)}_{L^\infty}|y-x|
\]
with an implicit constant independent of $t$. Hence, for any $N\geq n+1$
\[
	\abs{K(t,x,y)}\lesssim t \frac{t^{-n}}{(1+\abs{x-y}/t)^N},
\]
which yields that  $[P_t^{(2)},M_m]$ is a bounded operator in $L^1(\R^n)$ with norm bounded by $t$.
This fact and \eqref{eq:Q1} imply that
\[
	\abs{I}\leq \norm{F}_{\BMO}\norm{G}_{\BMO} \int_0^1 \int\abs{[P_t^{(2)},M_m] H(x)}\dd x \frac{\dd t}{t}\leq
	\norm{F}_{\BMO}\norm{G}_{\BMO}\norm{H}_{L^1}.
\]
Then
\begin{equation}\label{eq:step_1}
	\abs{\left<\PP_2(F,G),H\right>}\lesssim \norm{F}_{\BMO}\norm{G}_{X_w}\norm{H}_{H^1}.
\end{equation}
A similar analysis for $\PP_1$ yields
 \[
 	\begin{split}
	\left<\PP_1(F,G),H\right> &=\int \int_0^1 {\brkt{Q_t F} \brkt{P_t^{(1)} G}} [Q_t^{(2)},M_m] H\frac{\dd t}{t}\dd x+\\
	&+\int \int_0^1 {\brkt{Q_t F} \brkt{P_t^{(1)} G}}\big(Q_t^{(2)} H\big) m(t,\cdot)\frac{\dd t}{t}\dd x = A+B.
	\end{split}
\]
Taking $H=I_{1/\sigma_w} h\in H^{1}_{1/\sigma_w}$, proceeding as in the study of \eqref{eq:estimate_pi4} yields
\[
	\abs{B}\lesssim \norm{F}_{\BMO}\norm{G}_{X_w}\norm{h}_{H^1}.
\]
The first term can be written as
\[
	A=\int \int_0^1 \brkt{Q_t F} w(t) [Q_t^{(2)},M_m] H v(t,x)\frac{\dd t}{t}\dd x
\]
where $v(t,x)=\brkt{P_t^{(1)} G}(x)/w(t)$. Observe that $\sup_t \norm{v(t,\cdot)}_{L^\infty}\lesssim \norm{G}_{X_w}$ and also, with a similar argument as above,  $[Q_t^{(2)},M_m]$ is bounded in $L^1$ with norm bounded by a constant multiple of $t$. Then, \eqref{eq:Q1} and the hypothesis that $\int_0^1 w(t)\dd t$ is finite yield
\[
	\abs{A}\lesssim \norm{F}_{\BMO} \norm{G}_{X_w}\norm{H}_{L^1}.
\]
Therefore
\begin{equation}\label{eq:step_2}
		\left<\PP_1(F,G),I_{1/\sigma_w}h\right> \lesssim \norm{F}_{\BMO} \norm{G}_{X_w}\norm{h}_{H^1}
\end{equation}
 Hence, putting together \eqref{eq:step_1} and \eqref{eq:step_2} it follows that
 \[
 	\norm{\PP(F,G)}_{\BMO_{\sigma_w}} \lesssim \norm{F}_{\BMO} \norm{G}_{X_w}.
 \]
 Taking \eqref{eq:obs} into account and using \eqref{eq:bmo}, we have
 \[
 	\norm{\PP(F,G)}_{\BMO_{\sigma_w}} \lesssim \norm{(1-\widehat{\Theta})(D)F}_{\BMO} \norm{G}_{X_w}\lesssim \norm{F}_{\bmo}\norm{G}_{X_w}.
 \]
 \end{proof}

\begin{cor}\label{cor:bmo_bmo} Under the same conditions as in the previous theorem with $w(t)=1+\log_+ 1/t$ one has
\[
	\norm{\PP (f,g)}_{\BMO_{\sigma_w}} \leq C \norm{f}_{\bmo}\norm{g}_{\bmo}.
\]
\end{cor}
  \section{Application to one dimensional bilinear  Fourier integral operators }\label{FIOs}

	The investigations concerning bilinear Fourier integral operators started in \cite{GP}, followed by \cite{RS}. In \cite{RRS} the authors establish the regularity of bilinear Fourier integral operators with bilinear amplitudes in $S^m_{1,0} (n,2)$ and non-degenerate phase functions on $L^p$ spaces. However it was also shown that many endpoint results fail in dimension one. The following discussion i.e. Theorem \ref{thm:1d_FIO}, can be viewed as one positive result in that direction.

   \begin{defn}\label{defn of hormander symbols}
Let $N\in \N$. A function $\sigma \in \mathcal{C}^{\infty}(\R^n \times\R^{Nn})$ (called the amplitude in this context), belongs to the class $S^{0}_{1,0}(n,N)$, if for all multi-indices $\alpha$ and $\beta$, there exist constants $C_{\alpha,\beta}$ such that
   \begin{align*}
      |\partial_{\Xi}^{\alpha}\partial_{x}^{\beta}\sigma(x,\Xi)| \leq C_{\alpha,\beta} \brkt{1+\abs{\Xi}}^{m-\vert \alpha\vert}, \qquad \text{for all $(x,\Xi)\in \R^n\times \R^{Nn}$}.
   \end{align*}
\end{defn}

\begin{defn}\label{defn:phase} A function $\phase(x,\xi)\in \mathcal{C}^\infty\brkt{\R^n\times (\R^n\setminus \{0\})}$ is called a non-degenerate phase function if it is real-valued, positively homogeneous of degree one in $\xi$ and verifies
\[
	\det {\frac{\d^2 \phase}{\d x \d\xi}}\neq 0,
\]
for $\xi\neq 0$ and $x$ in the spatial support of the amplitude.
\end{defn}
\begin{defn}
A \emph{linear Fourier integral operator} $T^{\phase}_{\sigma}$ is an operator which is defined to act on a Schwartz function $f$ by the formula
      \begin{equation} \label{eq:Linear_FIO}
      	T_\sigma^{\phase}f(x)= \int_{\R^n}  \sigma(x,\xi) \widehat{f}(\xi)e^{i\phase(x,\xi)} \, \ddd\xi.
      \end{equation}
\end{defn}
In what follows, we would need the boundedness of Fourier integral operators on the space $\bmo$. Although the boundedness of these operators on $h^p$ for $p\leq 1$, was established by M. Peloso and S. Secco \cite{PS}, the $\bmo$ boundedness of the operators would not follow directly from their $h^1$ boundedness result and duality. This is because the adjoint of an operator of the type \eqref{eq:Linear_FIO}, although being a general Fourier integral operator with an amplitude in the same class, is not of the same form as \eqref{eq:Linear_FIO}. In particular, the amplitude of this operator would not be compactly supported in both spatial variables, and therefore one can not represent the adjoint as a finite sum of operators of the form \eqref{eq:Linear_FIO} with amplitudes that are compactly supported in $x$.

\begin{thm}\label{eq:thm_bmo} For any linear Fourier integral operator  $T_\sigma^{\phase}$ of the type \eqref{defn RS FIO} with an amplitude $\sigma(x,\xi)\in S^{-(n-1)/2}_{1,0}(\R^n)$ with spatial compact support one  has
\[
	\norm{T_\sigma^{\phase} f}_{\bmo}\leq C\norm{f}_{\bmo}.
\]
\end{thm}
\begin{proof} First we claim that for any $\sigma$ and $\phase$ as above
\begin{equation*}\label{eq:claim}
	\norm{S_{\sigma}^\phase f}_{L^1(\R^n)}\leq \norm{f}_{h^1(\R^n)},
\end{equation*}
where
$
	S_{\sigma}^\phase f(x)=\int \brkt{\int \sigma(y,\xi) e^{-i\phase(y,\xi)+ix\cdot\xi} \ddd \xi}  f(y)\dd y=\int K(x,y) f(y) \dd y.
$

Assuming this claim, since
\[
	r_j (S_{\sigma}^\phase )f(x)=\int \brkt{\int \frac{\xi_j}{\abs{\xi}}(1- \widehat{\Theta }(\xi))\sigma(y,\xi) e^{-i\phase(y,\xi)+ix\cdot\xi} \ddd \xi}  f(y)\dd y,
\]
{with $r_j$ as in \eqref{eq:local_Riesz}}, and since $r_j (S_{\sigma}^\phase )$ is an operator of the same type as $S_{\sigma}^\phase$ with the amplitude $$\frac{\xi_j}{\abs{\xi}}(1- \widehat{\Theta }(\xi))\sigma(y,\xi) \in S^{-(n-1)/2}_{1,0}(\R^n),$$ it follows that for any $j=1,\ldots, n$
$
	\norm{r_j (S_{\sigma}^\phase f) }_{L^1(\R^n)}\leq C \norm{f}_{h^{1}(\R^n)}.
$
\\
Now \eqref{eq:norm_h1} yields $\norm{S_{\sigma}^\phase f }_{h^1(\R^n)}\leq C \norm{f}_{h^{1}(\R^n)}$ . Therefore using duality, since the adjoint of $T^\phase_\sigma$ is $S^\phase_\sigma$, the $\bmo$ boundedness follows.

Using the atomic characterisation of $h^1(\R^n)$ as done in \cite{Gol}*{Lemma 4}, it is enough to prove the result for an $h^1(\R)$ atom, which is a  function $a$ supported in a ball $B$, such that $\abs{a}\leq \abs{B}^{-1}$, and $\int a \dd x=0$ if the radius of $B$ is smaller than one, and with no  moment conditions on $a$ in case  the radius of $B$ is bigger than one. In particular, if the radius of $B$ is smaller than one, $a$ is also an $H^1(\R^n)$ atom. Therefore \cite{Sogge}*{Proposition 6.2.6}  implies 
\[
	\norm{S_{\sigma}^\phase a}_{L^1(\R^n)}\lesssim 1,
\]
whenever the radius of the support of $a$ is less than one. If the radius is larger than one we proceed as follows. As is observed in \cite{Sogge}*{p. 179}, integration by parts yields that the kernel of  $S_{\sigma}^\phase$, $K(x,y)$ satisfies the estimate $\abs{K(x,y)}\lesssim \abs{x}^{-N}$ for any $N$ when $\abs{x}$ is larger than some fixed constant $r$ depending on $\sigma$ and $\phase$. Therefore, 
\[
	\int \abs{S^\phase_\sigma a(x)}\dd x=\brkt{\int_{\abs{x}\leq r}+\int_{\abs{x}\geq r}} \abs{S^\phase_\sigma a(x)}\dd x=I+II.
\]
Now, $II\lesssim \int_{\abs{x}\geq r} \abs{x}^{-n-1}\dd x \int \abs{a(y)}\dd y\lesssim 1$. For $I$, we simply use the compact $x$-support of the domain of integration, the Cauchy-Schwarz inequality and the $L^2$ boundedness of Fourier integral operators of order zero to conclude that $I\lesssim 1$.
\end{proof}

\begin{defn}
A {bilinear oscillatory integral operator} $T^{\phase_1,\phase_2}_{\sigma}$ is an operator which is defined to act on Schwartz functions $f$ and $g$ by the formula
\begin{equation}\label{defn RS FIO}
T^{\phase_1,\phase_2}_{\sigma}(f,g)(x) = \iint \sigma(x,\xi,\eta)\widehat{f}(\xi)\widehat{g}(\eta)e^{i\phase_1(x,\xi)+i \phase_2(x,\eta)} \ddd\xi \ddd\eta.
\end{equation}
\end{defn}

  \begin{thm}\label{thm:1d_FIO} If $\sigma\in S^0_{1,0}(1,2)$ is compactly supported in the spatial variable and $\phase_1,\phase_2$ are non-degenerate phase functions. Then there exists a constant $C$, depending only on a finite number of derivatives of $\sigma$, such that
\begin{equation*} %
\|T_\sigma^{\phase_1,\phase_2}(f,g)\|_{\BMO_{\sigma_w}}\leq C \|f\|_{L^\infty}\|g\|_{L^\infty},
\end{equation*}
with $w(t)=1+\log_+ 1/t$.
 \end{thm}
 \begin{proof} It is shown in \cite{RRS2}*{Section 4} that the study of the operator  $T_\sigma^{\phase_1,\phase_2}$ reduces to study an operator of the form
 \[
 	\PP(T^{\phase_1}_{\nu} f,T^{\phase_2}_{\nu} g)+E(f,g),
 \]
 where $\nu(x,\xi)=\chi(x)\mu(\xi)$, $\chi$ is a compactly supported smooth function and $\mu$ is a smooth function such that $1-\mu$ is supported in a neighbourhood of the origin and $E:L^\infty \times L^\infty \to L^\infty$.

Observe that in particular, $\nu \in S^0_{1,0}(1,1)$ with compact support in $x$. Then, Proposition \ref{eq:thm_bmo} implies in particular that $T^{\phase_1}_{\nu}: L^\infty\to \bmo$. So, taking into account Corollary \ref{cor:bmo_bmo}, the result follows by composition of bounded operators and triangle inequality.
 \end{proof}
 \section{Application to bilinear Coifman-Meyer multipliers and Kato-Ponce estimates}\label{kato-ponce}

 In this section, following the method of Coifman-Meyer \cite{CM4} we can also produce a boundedness result for bilinear Coifman-Meyer multipliers. Furthermore, we can apply this to obtain certain Kato-Ponce type estimates, which originated in the seminal paper by T. Kato and G. Ponce \cite{KPo}.

 \begin{thm}\label{thm:H-M} For any $\sigma(\xi,\eta)\in \mathcal{C}^\infty(\R^n\times \R^n\setminus\{0,0\})$ such that
 \begin{equation}\label{eq:H_M}
 	 \abs{\d^\alpha_\xi \d^\beta_\eta \sigma(\xi,\eta)}\lesssim \brkt{\abs{\xi}+\abs{\eta}}^{-\abs{\alpha}-\abs{\beta}}, \quad {\rm for}\, (\xi,\eta)\neq (0,0)
 \end{equation}
 for any $\alpha$ and $\beta$, then the bilinear operator
 \[
 	\sigma(D)(f,g)(x)=\iint \sigma(\xi,\eta) \widehat{f}(\xi)\widehat{g}(\eta) e^{i x(\xi+\eta)}\ddd \xi \ddd \eta
 \]
 satisfies, for any admissible weight $w$
 \[
 	\norm{\sigma(D)(f,g)}_{\BMO_{\sigma_w}}\lesssim \norm{f}_{X_w}\norm{g}_{X_w}.
 \]
 \end{thm}
 \begin{proof} Proceeding as in the proof of Proposition 2 in \cite{CM4}*{p.154}, one decomposes $\sigma=\sigma_1+\sigma_2$ with $\sigma_1(\xi,\eta)$ supported in $\abs{\xi}\geq \abs{\eta}/20$, $\sigma_2(\xi,\eta)$ supported in $10\abs{\xi}\leq \abs{\eta}$ and $\sigma_1,\sigma_2$ satisfying \eqref{eq:H_M}.

 Let $\widehat{\psi}$ be a smooth Schwartz function supported in $4\leq 5 \abs{\xi}\leq 6$ and such that
 $$\int_0^\infty \abs{\widehat{\psi}(\xi)}^2\frac{\dd t}{t}=1, \qquad \mathrm{for}\,\,\, \xi\neq 0.$$
 Let $\widehat{\phi}$  be a smooth compactly supported Schwartz function such that is equal to $1$ if $\abs{\xi}\leq 30$. Then an argument based on integration by parts shows that for any positive real number $N$  
one can write
 \[
 	\sigma_1(D)(f,g)(x)=\iint  \Pi_{u,v}(f,g)(x)\frac{\dd u \dd v}{(1+\abs{u}^2+\abs{v}^2)^N},
 \]
 where 
 \[
 	\Pi_{u,v}(f,g)(x)=\int_0^\infty Q_t^u(f) P_t^v(g) m(t,u,v) \frac{\dd t}{t},
 \]
$\sup_{t,u,v} \abs{m(t,u,v)}<+\infty$ and $$Q_t^u(f)=\widehat{\psi^u}(tD) f,\quad P_t^u(f)=\widehat{\phi^v}(tD) g$$ with $\psi^u(x)=\psi(x+u)$ and $\phi^v(x)=\phi(x+v)$. Another integration by parts argument revelas that, for any $\delta>0$, 
\[
	\abs{\psi^u(x)}\lesssim \frac{(1+\abs{u})^{n+\delta}}{(1+\abs{x})^{n+\delta}}\quad {\rm and} \quad \abs{\phi^v(x)}\lesssim \frac{(1+\abs{v})^{n+\delta}}{(1+\abs{x})^{n+\delta}}.
\]
So one can apply Theorem \ref{thm:main_thm} to deduce
\[
	\norm{\Pi_{u,v}(f,g)}_{\BMO_{\sigma_w}}\lesssim P(u,v) \norm{f}_{X_w}\norm{g}_{X_w},
\]
with $P(u,v)$ a polynomial expression in $\abs{u}$ and $\abs{v}$, independent on $f,g$. Then, taking $N$ large enough yields
\[
	\norm{\sigma_1(D)(f,g)}_{\BMO_{\sigma_w}}\lesssim \norm{f}_{X_w}\norm{g}_{X_w}.
\]
For $\sigma_2(D)$ the roles of $\xi$ and $\eta$ are reversed, but by the symmetry of the estimates, a similar argument to the previous one yields
\[
	\norm{\sigma_2(D)(f,g)}_{\BMO_{\sigma_w}}\lesssim \norm{f}_{X_w}\norm{g}_{X_w}.
\]
Then, the result follows by putting these last two estimates together.
    \end{proof}

\begin{rem}\label{eq:final_remark}  One can weaken the assumption \eqref{eq:H_M} on the symbol by requiring only finite number of derivatives of the symbol. Tracing the proofs of theorems \ref{thm:main}  and \ref{thm:H-M}, it turns out that we would need at least $4n+1$ derivatives. In light of the results of \cite{GT} and \cite{MT,Tomita}, there seems to be some room for improvement of the number of derivatives.
\end{rem}

\begin{cor} If $\sigma(\xi,\eta)\in \mathcal{C}^\infty(\R^n\times \R^n\setminus\{0,0\})$ satisfying \eqref{eq:H_M}, then
\[
	\norm{\sigma(D)(f,g)}_{\BMO_{\sigma_w}}\lesssim \norm{f}_{\bmo}\norm{g}_{\bmo}.
\]
with $w=1+\log_+ 1/t$ and
\[
		\norm{\sigma(D)(f,g)}_{\BMO}\lesssim \norm{f}_{L^\infty}\norm{g}_{L^\infty}.
\]
\end{cor}

\begin{cor} For any $f,g\in \bmo$
\[
	\norm{f g}_{\BMO_{\sigma_w}}\lesssim \norm{f}_{\bmo}\norm{g}_{\bmo},
\]
with $w=1+\log_+ 1/t$.
\end{cor}
\begin{proof} This follows from the previous corollary by taking $\sigma(\xi,\eta)=1$.
\end{proof}
As a consequence of Theorem \ref{thm:H-M} and Remark \ref{eq:final_remark}  we obtain the following endpoint Kato-Ponce type inequality, which is a variant of \cite{GMN}*{Theorem 3}.

\begin{thm}\label{thm:kato_ponce} If $s >4n+1$ and $w$ is an admissible weight, then for every $f, g \in \mathcal{S}$
\begin{equation*}
\norm{D^s(fg)}_{\BMO_{\sigma_w}}  \lesssim \norm{D^{s}f}_{X_w} \norm{g}_{X_w} + \norm{f}_{X_w} \norm{D^{s}g}_{X_w},
\end{equation*}
where $D^s$ denotes the operator defined for $h \in \mathcal{S}$ as
$
\widehat{D^sh}(\xi):= |\xi|^s \hat{h}(\xi)$ for all $\xi \in \R^n$.
In particular, for $w=1+\log_+ 1/t$
\begin{equation*}
\norm{D^s(fg)}_{\BMO_{\sigma_w}}  \lesssim \norm{D^{s}f}_{\bmo} \norm{g}_{\bmo} + \norm{f}_{\bmo} \norm{D^{s}g}_{\bmo}.
\end{equation*}
\end{thm}
\begin{proof} Following the approach in \cite{GOh}, it is shown in \cite{GMN} that the bilinear mapping $(f,g) \mapsto D^s(fg)$ can be decomposed into the sum of three bilinear multipliers
\begin{equation*}
D^s(fg) = \sigma_{1,s}(D)(D^s f,g) + \sigma_{2,s}(D)(f,D^s g) + \sigma_{3,s}(D)(f,D^sg),
\end{equation*}
where, for $s > 4n+1$, these have symbols satisfying \eqref{eq:H_M} {for $\abs{\alpha}+\abs{\beta}\leq 4n+1$}. Then the result follows from Theorem \ref{thm:H-M} and Remark \ref{eq:final_remark}.

The last assertion follows by taking the particular weight $w=1+\log_+ 1/t$ and applying Corollary \ref{cor:bmo}.
\end{proof}

\begin{rem} Observe that for $w=1$, the previous theorem recovers \cite{GMN}*{Theorem 3} albeit with a larger amount of derivatives required in \eqref{eq:H_M}.
 \end{rem}

\end{document}